\documentclass[11pt,reqno]{amsart}
\usepackage[margin=1in]{geometry}

\usepackage{amssymb,amsmath,amsthm}
\usepackage{microtype}
\usepackage{mathtools}
\usepackage{enumitem}

\usepackage[colorlinks=true, pdfstartview=FitH, linkcolor=blue, citecolor=blue, urlcolor=blue]{hyperref}





\newcommand{\eps}{\varepsilon}
\newcommand{\defeq}{\vcentcolon=}

\newcommand{\E}{\mathbb E}
\newcommand{\Var}{\mathrm{Var}}

\renewcommand{\leq}{\leqslant}
\renewcommand{\le}{\leqslant}
\renewcommand{\geq}{\geqslant}
\renewcommand{\ge}{\geqslant}

\newcommand{\N}{\mathbb{N}}
\newcommand{\R}{\mathbb{R}} 
\newcommand{\Z}{\mathbb{Z}}

\newtheorem{theorem}{Theorem}[section]

\newtheorem{lemma}[theorem]{Lemma}

\theoremstyle{definition}

\newtheorem{remark}[theorem]{Remark}

\numberwithin{theorem}{section}

\allowdisplaybreaks

\begin{document}

\title{The extensible no-four-on-a-circle problem}

\author{Anubhab Ghosal}
\author{Ritesh Goenka}

\address{Mathematical Institute \\ University of Oxford \\ Oxford, UK OX2 6GG}
\email{\{ghosal,goenka\}@maths.ox.ac.uk}


\begin{abstract}
    We show that there exists a set $S \subset \mathbb{Z}^2$ containing no four points on a circle or a line such that $|S \cap [n]^2| = \Omega(n)$ as $n \rightarrow \infty$. Since any no-four-on-a-circle set in $[n]^2$ has size $O(n)$, this resolves (up to a constant) a question raised by the current authors and Keevash concerning the density of extensible no-four-on-a-circle constructions. Our construction is based on weighted random sampling from the integer lattice followed by careful deletion.
\end{abstract}

\maketitle

\section{Introduction}

Determining the maximum size of a subset of $[n]^2$ lacking specific geometric configurations has a rich history in discrete geometry (see the recent survey~\cite{JNNS}), tracing back to the classical no-three-in-line problem introduced by Dudeney~\cite{Dud}. More recently, \emph{extensible} versions of such questions have been considered, where one seeks subsets of $\mathbb{Z}^2$ of large (appropriately scaled) density avoiding these geometric configurations. For instance, Erde~\cite[Problem 5]{Open} asked if there is a set $S \subset \mathbb{Z}^2$ without collinear triples of points satisfying $|S \cap [n]^2| = \Omega(n)$ for all sufficiently large $n$. All known constructions of no-three-in-line subsets of $[n]^2$ with $\Omega(n)$ points are derived from subsets of $\mathbb{F}_p^2$ with $p \approx n$, and therefore cannot be easily extended to the entire integer lattice. See~\cite[Problem 72]{Green} for related discussion and \cite{Gho,NNW} for partial progress on this problem. See also \cite{GJMN} for the extensible version of the no-$(k+1)$-in-line problem for general $2 \le k \le n$.

Erd\H{o}s and Purdy asked the no-four-on-a-circle problem: determine the maximum number $f(n)$ of points that can be chosen from $[n]^2$ with no four of them concyclic or collinear. It is known that $f(n) = \Theta(n)$, with the current best-known lower bound $f(n) \ge 2n - o(n)$ due to the present authors, Grebennikov, Keevash, Kwan, and Pham~\cite{GGGKKP}, and upper bound $f(n) \le 5n/2 - 3/2$ due to Thiele~\cite{Thi2}. Similarly to the no-three-in-line problem, until recently, all known constructions demonstrating $f(n) = \Omega(n)$ were algebraic in nature~\cite{DX,Thi}. The current authors and Keevash~\cite{GGK} provided the first randomised construction of size $\Omega(n)$ for this problem and proposed studying the related extensible version. In this paper, we prove the following theorem, which resolves the extensible no-four-on-a-circle problem up to a constant.

\begin{theorem}
\label{thm:main}
    There exist constants $c > 0$, $n_0 \in \N$, and a set $S \subset \Z^2$ such that no four points of $S$ lie on a circle or a line and $|S \cap [n]^2| \ge cn$ for all $n \ge n_0$.
\end{theorem}

At a high level, our construction follows the same strategy as the recent work by the first-named author~\cite{Gho}. We first select each point $x \in \mathbb{N}^2$ independently with probability $\alpha/\|x\|_\infty$ for some small constant $\alpha > 0$, to obtain a random set $Q$. We then delete the point with the largest infinity norm from each concyclic or collinear quadruple of points in $Q$ to obtain the set $S$. Finally, we show that the set $S$ so obtained satisfies the conclusion of Theorem~\ref{thm:main} with positive probability for some small enough absolute constant $c > 0$.

The proof relies on expectation and variance estimates for the number of concyclic quadruples of selected points in $[n]^2$. For the case of uniform weights, this expectation was already estimated in \cite{GGK}. The analysis there proceeds by grouping concyclic quadruples into two kinds: ones that come from isosceles trapezia and ones that do not. The contribution of the former was carefully estimated, whereas that of the latter was shown to be a lower-order term using a number-theoretic result of Huxley and Konyagin~\cite{HK}. Our analysis follows a similar template, albeit the computations are more technical due to the non-uniformity of weights.

\begin{remark}
    Theorem~\ref{thm:main} also gives a linear lower bound for the extensible no-four-in-line problem. More generally, for every fixed integer $k \ge 3$, there exists a set $S_k \subset \Z^2$ with no $k+1$ points on a line such that $|S_k \cap [n]^2| = \Omega(kn)$ as $n \rightarrow \infty$. Indeed, let $r = \lfloor k/3 \rfloor$ and $e_1 = (1,0)$, and take
    \begin{equation*}
        S_k = \bigcup_{j=0}^{r-1} (S + je_1),
    \end{equation*}
    where $S$ is the set in Theorem~\ref{thm:main}. Every line meets $S_k$ in at most $3r\le k$ points. Moreover, each point belongs to at most three of these translates, since otherwise we would get four points in $S$ that lie on a horizontal line. Consequently, we obtain
    \begin{equation*}
        |S_k \cap [n]^2| \ge \frac{r}{3} |S \cap [n-r+1]^2| \ge \frac{cr}{3} (n-r+1) = \Omega(kn)
    \end{equation*}
    as $n \rightarrow \infty$. This matches the elementary upper bound $kn$ up to an absolute constant. G\'abriel, J\'anosik, Melj\'an, and N\'ador~\cite{GJMN} recently improved upon this trivial upper bound by showing that any extensible no-$(k+1)$-in-line set $S \subset \Z^2$ satisfies
    \begin{equation*}
        \liminf_{n \rightarrow \infty} \frac{|S \cap [n]^2|}{n} < 0.897 k.
    \end{equation*}
\end{remark}

The remainder of this paper is organised as follows. In Section~\ref{sec:proof}, we prove the main result conditional on the aforementioned expectation and variance estimates. We then prove the relevant expectation and variance estimates for collinear quadruples, symmetric concyclic quadruples, and asymmetric concyclic quadruples in Sections~\ref{sec:lin}, \ref{sec:sym}, and \ref{sec:asym}, respectively.

\vspace{8pt}

\paragraph{\textbf{Asymptotic notation.}}
For nonnegative quantities $a$ and $b$, we write $a = O(b)$ or $a \lesssim b$ if $a \le Cb$ for some absolute constant $C > 0$, uniformly over the stated ranges of the parameters. We write $a = \Omega(b)$ or $a \gtrsim b$ if $b = O(a)$, and $a = \Theta(b)$ or $a \asymp b$ if both
$a = O(b)$ and $b = O(a)$. Dependence of the implicit constant on some parameters is indicated by subscripts, as in $O_\varepsilon$ or $\lesssim_\varepsilon$. When a limit such as $n \to \infty$ is specified, the corresponding comparisons are required to hold only for all sufficiently large $n$, and the threshold may depend on parameters held fixed in that limit.

\vspace{8pt}

\paragraph{\textbf{Statement of AI use.}}
ChatGPT~5.4--5.6 Pro and Gemini~3.1 Pro were used to assist the authors in carrying out and simplifying the computations required to prove the various parts of Lemma~\ref{lem:main}. In addition, ChatGPT~6 Pro was used to polish the final draft of the paper. The high-level proof ideas, including those behind Lemma~\ref{lem:main}, are entirely due to the authors. The authors assume full responsibility for the contents of the paper.

\section*{Acknowledgments}

AG is supported by a joint Clarendon Fund and Oxford Ryniker Lloyd Graduate Scholarship. RG is supported by a joint Clarendon Fund and Exeter College SKP scholarship.

\section{Proof of the main result}
\label{sec:proof}

We begin with a definition. A \emph{bad quadruple} refers to an unordered $4$-tuple of distinct points in $\mathbb{Z}^2$ that is either collinear or concyclic. We partition bad quadruples into three classes:
\begin{itemize}
    \item collinear quadruples,
    \item symmetric cyclic quadruples, and
    \item asymmetric cyclic quadruples.
\end{itemize}
Symmetric cyclic quadruples correspond to vertices of isosceles trapezia, while asymmetric cyclic quadruples correspond to vertices of other cyclic quadrilaterals.

For $\alpha \in (0, 1)$, let $Q \defeq Q(\alpha)$ be a random subset of $\mathbb{Z}^2$ obtained by selecting each point $x \in \mathbb{N}^2$ independently with probability $p_x \defeq \alpha/\|x\|_\infty$. Further, for an integer $T \ge 0$, let $B_T$ and $R_T$ denote the dyadic box and dyadic shell defined by
\begin{equation*}
    B_T \defeq [2^T]^2 \quad \text{ and } \quad R_T \defeq B_T \setminus B_{T-1},
\end{equation*}
respectively, where $B_{-1}$ is taken to be the empty set. Let $X_T \defeq X_T(\alpha)$ denote the number of points in $Q \cap B_T$. In the following lemma, we record basic expectation and variance estimates for $X_T$.

\begin{lemma}
\label{lem:Xcount}
    For $\alpha \in (0, 1)$ and integer $T \ge 0$, we have $\alpha 2^T \le \E X_T \le \alpha 2^{T+1}$ and $\Var(X_T) \le \E X_T$.
\end{lemma}

\begin{proof}
    For each $r \in \N$, there are $2r-1$ points $x \in \N^2$ with $\|x\|_\infty = r$. Therefore, we have
    \begin{equation*}
        \E|Q \cap B_T| = \sum_{x \in B_T} \frac{\alpha}{\|x\|_\infty} = \sum_{r=1}^{2^T} (2r-1) \frac{\alpha}{r} = \alpha \sum_{r=1}^{2^T} \left(2 - \frac{1}{r}\right) \in [ \alpha 2^T, \alpha 2^{T+1}].
    \end{equation*}
    The variance estimate holds because $X_T$ can be expressed as a sum of independent indicator random variables $(\mathbf{1}\{x \in Q\})_{x \in B_T}$.
\end{proof}

Let $Y_T^{\mathrm{lin}} \defeq Y_T^{\mathrm{lin}}(\alpha)$, $Y_T^{\mathrm{sym}} \defeq Y_T^{\mathrm{sym}}(\alpha)$, and $Y_T^{\mathrm{asym}} \defeq Y_T^{\mathrm{asym}}(\alpha)$ denote the number of collinear quadruples, symmetric cyclic quadruples, and asymmetric cyclic quadruples in $Q \cap B_T$, respectively. In the following lemma, we provide expectation and variance estimates for these random variables.

\begin{lemma}
\label{lem:main}
    For $\alpha \in (0, 1)$ and integer $T \ge 0$, we have
    \begin{enumerate}
        \item $\E Y_T^{\mathrm{lin}} = O(\alpha^4 2^T)$,
        \item $\Var(Y_T^{\mathrm{lin}}) = O((T+1)^4 2^T)$,
        \item $\E Y_T^{\mathrm{sym}} = O(\alpha^4 2^T)$,
        \item $\Var(Y_T^{\mathrm{sym}}) = O((T+1)^4 2^T)$, and
        \item $\E Y_T^{\mathrm{asym}} = O(\alpha^4 2^{0.98 T})$.
    \end{enumerate}
\end{lemma}

As mentioned in the introduction, we will prove parts (a) and (b), (c) and (d), and (e) of the above lemma in Sections~\ref{sec:lin}, \ref{sec:sym}, and \ref{sec:asym}, respectively. We now prove our main result assuming Lemma~\ref{lem:main}.

\begin{proof}[Proof of Theorem~\ref{thm:main}]
    We may assume that all the estimates in Lemma~\ref{lem:main} hold with the same universal constant $C \ge 1$. Fix an arbitrary $0 < \alpha < 1/(12C)^{1/3}$. For $T \in \Z_{\ge 0}$, let $\mathcal{E}_T$ be the event
    \begin{equation*}
        \{X_T \ge \alpha 2^{T-1}\} \cap \{Y_T^{\mathrm{lin}} \le 2C \alpha^4 2^T\} \cap \{Y_T^{\mathrm{sym}} \le 2C \alpha^4 2^T\} \cap \{Y_T^{\mathrm{asym}} \le 2C \alpha^4 2^T\}.
    \end{equation*}
    Applying Chebyshev's inequality to the random variable $X_T$ and using Lemma~\ref{lem:Xcount}, we obtain
    \begin{equation}
    \label{eqn:1}
        \mathbb{P} (X_T < \alpha 2^{T-1}) \le \mathbb{P} \left(X_T < \frac{1}{2} \E X_T\right) \le \frac{4}{\E X_T} \le \frac{4}{\alpha 2^T}.
    \end{equation}
    Similarly, applying Chebyshev's inequality to $Y_T^{\mathrm{lin}}$ and using Lemma~\ref{lem:main}~(a) and (b), we obtain
    \begin{equation}
    \label{eqn:2}
        \mathbb{P} (Y_T^{\mathrm{lin}} > 2C \alpha^4 2^T) \le \mathbb{P} (Y_T^{\mathrm{lin}} > \E Y_T^{\mathrm{lin}} + C \alpha^4 2^T) \le \frac{C (T+1)^4 2^T}{(C \alpha^4 2^T)^2} = \frac{(T+1)^4}{C \alpha^8 2^T}.
    \end{equation}
    Similarly, applying Chebyshev's inequality to $Y_T^{\mathrm{sym}}$ and using Lemma~\ref{lem:main}~(c) and (d), we obtain
    \begin{equation}
    \label{eqn:3}
        \mathbb{P} (Y_T^{\mathrm{sym}} > 2C \alpha^4 2^T) \le \frac{(T+1)^4}{C \alpha^8 2^T}.
    \end{equation}
    Finally, applying Markov's inequality to the random variable $Y_T^{\mathrm{asym}}$ and using Lemma~\ref{lem:main}~(e) yields
    \begin{equation}
    \label{eqn:4}
        \mathbb{P} (Y_T^{\mathrm{asym}} > 2C \alpha^4 2^T) \le \frac{C \alpha^4 2^{0.98 T}}{2C \alpha^4 2^T} = \frac{1}{2^{0.02 T + 1}}.
    \end{equation}
    Taking a union bound using \eqref{eqn:1}--\eqref{eqn:4}, we obtain
    \begin{equation*}
        \mathbb{P}(\mathcal{E}_T^c) \le \frac{4}{\alpha 2^T} + \frac{2(T+1)^4}{C \alpha^8 2^T} + \frac{1}{2^{0.02 T + 1}}.
    \end{equation*}
    The above inequality implies that $\sum_{T\geq 0}\mathbb{P}(\mathcal{E}_T^c)<\infty$. In particular, there exists a nonnegative integer $T_\alpha$ depending only on $\alpha$ so that $\sum_{T\geq T_\alpha}\mathbb{P}(\mathcal{E}_T^c)<1$ and therefore
    \begin{equation*}
        \mathbb{P}\left(\bigcap_{T = T_\alpha}^\infty  \mathcal{E}_T\right) > 0.
    \end{equation*}
    Thus we may fix a realisation of $Q$ such that the event $\mathcal{E}_T$ holds for each $T \ge T_\alpha$. Let $S$ be the set obtained from $Q$ by deleting the point with the largest infinity norm (with arbitrary tie breaking) from each bad quadruple in $Q$. Then, by construction, $S$ does not contain any concyclic or collinear quadruples of points. Let $n \ge 2^{T_\alpha}$ be any integer and set $T_n = \lfloor \log_2 n \rfloor$. Note that $S$ satisfies
    \begin{align*}
        |S \cap [n]^2| \ge |S \cap B_{T_n}| &\ge X_{T_n} - Y_{T_n}^{\mathrm{lin}} - Y_{T_n}^{\mathrm{sym}} - Y_{T_n}^{\mathrm{asym}}\\
        &\ge \alpha 2^{T_n-1} - 2C \alpha^4 2^{T_n} - 2C \alpha^4 2^{T_n} - 2C \alpha^4 2^{T_n}\\
        &= \alpha (1 - 12C \alpha^3) 2^{T_n-1}\\
        &\ge \frac{\alpha}{4} (1 - 12C \alpha^3) n.
    \end{align*}
    Finally, observing that $\frac{\alpha}{4} (1 - 12C \alpha^3) > 0$ since we chose $\alpha < 1/(12C)^{1/3}$ finishes the proof.
\end{proof}

\section{Collinear quadruples}
\label{sec:lin}

We begin with some definitions. A nonzero vector $v = (a,b) \in \Z^2$ is called \emph{primitive} if $\gcd(a,b) = 1$. A \emph{lattice line} is a line in $\mathbb{R}^2$ that contains at least two points from the lattice $\mathbb{Z}^2$.  Each lattice line is parallel to a primitive vector, unique up to sign: we fix one representative for each unoriented primitive direction, and call this vector the \emph{direction} of the line. For an integer $T \ge 0$, let $\mathcal{L}_T$ denote the set of collinear quadruples of points in $B_T$.

\begin{proof}[Proof of Lemma~\ref{lem:main}~(a)]
    By linearity of expectation and independence of $(\mathbf{1}\{x \in Q\})_{x \in B_T}$, we have
    \begin{equation}
        \label{eq:EYTlin}
        \E[Y_T^{\text{lin}}] = \sum_{\{x_1, x_2, x_3, x_4\} \in \mathcal{L}_T} p_{x_1}p_{x_2}p_{x_3}p_{x_4}.
    \end{equation}
    Fix $0 \le t_1 \le t_2 \le t_3 \le t_4 \le T$. We first bound the total contribution to the sum in \eqref{eq:EYTlin} from terms such that $x_i \in R_{t_i}$ for $i \in [4]$. There are $|R_{t_1}| \le 2^{2t_1}$ choices for $x_1$. Fix a point $x_1 \in R_{t_1}$. Let $m \in \N$ and let $L$ be a lattice line through $x_1$ in direction $v$ with $\|v\|_\infty = m$. We wish to bound
    \begin{equation}
    \label{eq: S(x,L) defn}
        S(x_1,L) \defeq \sum p_{x_2} p_{x_3} p_{x_4},
    \end{equation}
    where the sum is over triples $\{x_2, x_3, x_4\} \subseteq B_T \setminus \{x_1\}$ lying on the line $L$ with $x_i\in R_{t_i}$ for $2\leq i\leq 4$. Let $\{x_2, x_3, x_4\}$ be any such triple. Then, we have $m \le \|x_2 - x_1\|_\infty \le 2^{t_2}$. Therefore, the number of choices for $x_2$ is at most $|L\cap R_{t_2}| \le \lceil\frac{2^{t_2}}{m}\rceil \lesssim \frac{2^{t_2}}{m}$. Similarly, the number of choices for $x_i$ is at most $O(\frac{2^{t_i}}{m})$ for $i = 3, 4$. Since $p_{x_i} = O(\frac{\alpha}{2^{t_i}})$ for $i \in [4]$, we deduce that
    \begin{equation*}
        S(x_1,L) = O(\alpha^3 m^{-3}).
    \end{equation*}
    Now since there are $O(m)$ choices for the direction $v$ and hence the line $L$, with $\|v\|_{\infty} = m$, it follows that the total contribution of terms such that $x_i\in R_{t_i}$ for $i\in [4]$ to the sum in \eqref{eq:EYTlin} is at most
    \begin{equation*}
        O\left(2^{2t_1}\cdot \frac{\alpha}{2^{t_1}}\cdot \sum_{m=1}^\infty m\cdot \frac{\alpha^3}{m^3}\right) = O(\alpha^4 2^{t_1}).
    \end{equation*}
    Finally, since for a given $t_1$, there are at most $(T-t_1+1)^3$ choices for $(t_2,t_3,t_4)$, it follows that
    \begin{equation*}
        \E[Y_T^{\text{lin}}] \lesssim \sum_{t_1=0}^T\alpha^4 2^{t_1}(T-t_1+1)^3\lesssim \alpha^42^T,
    \end{equation*}
    where the last inequality follows from the convergence of the series $\sum_{t=0}^\infty 2^{-t}(t+1)^3$.
\end{proof}

For $x \in B_T$, define the quantity
\begin{equation*}
    \Gamma_T(x) \defeq \sum_{\{x, x_2, x_3, x_4\} \in \mathcal{L}_T} p_{x_2} p_{x_3} p_{x_4}.
\end{equation*}
In order to prove Lemma~\ref{lem:main} (b), we prove the following pointwise bound on $\Gamma_T(x)$ first.

\begin{lemma}
\label{lem:onefix}
    For each $x \in B_T$, we have
    \begin{equation*}
        \Gamma_T(x) = O(\alpha^3 (T+1)^4).
    \end{equation*}
\end{lemma}

\begin{proof}
    The proof is very similar to that of Lemma~\ref{lem:main}~(a), so we shall keep details to a minimum. Fix $0 \le t_2 \le t_3 \le t_4 \le T$. We define $S(x, L)$ as in \eqref{eq: S(x,L) defn}. The only difference is that now we have a slightly weaker bound on $m$, namely, $m \le \|x_3 - x_2\|_\infty \le 2^{t_3}$. As before, the number of choices for $x_i$ is $O(\frac{2^{t_i}}{m})$ for $i \in \{3,4\}$. However, we now use the weaker bound of $O(2^{t_2})$ on the number of choices for $x_2$. It follows that
    \begin{equation*}
        S(x, L) \lesssim 2^{t_2}\cdot\frac{2^{t_3}}{m}\cdot\frac{2^{t_4}}{m}\cdot\alpha 2^{-t_2}\cdot\alpha2^{-t_3}\cdot\alpha 2^{-t_4}= \alpha^3 m^{-2}.
    \end{equation*}
    Hence, the contribution to the sum defining $\Gamma_T(x)$ of terms with $x_i \in R_{t_i}$ for $i \in \{2,3,4\}$ is at most 
    \begin{equation*}
        O\left(\sum_{m=1}^{2^{t_3}} m\cdot \frac{\alpha^3}{m^2}\right) \lesssim \alpha^3 (T+1).
    \end{equation*}
    Noting that there are $O((T+1)^3)$ choices for the triple $(t_2,t_3,t_4)$, we obtain the claim.
\end{proof}

\begin{proof}[Proof of  Lemma~\ref{lem:main} (b)]
    Note that $ Y_T^{\mathrm{lin}} = \sum_{\tau \in \mathcal{L}_T} I_\tau$, where $I_\tau$ is the indicator that all four points in the quadruple $\tau$ are selected. It follows that
    \begin{align}
    \label{eq: second moment}
        \Var(Y_T^{\mathrm{lin}}) = \sum_{\tau, \tau'} \operatorname{Cov}(I_\tau,I_{\tau'})=\sum_{\tau \cap \tau'\neq \emptyset} \operatorname{Cov}(I_\tau,I_{\tau'})\leq \sum_{\tau \cap \tau'\neq \emptyset} \E(I_\tau I_{\tau'})=\sum_{i=1}^4V_i,
    \end{align}
    where, for $i\in [4]$, $V_i$ is defined as
    \begin{equation*}
        V_i \defeq \sum_{|\tau \cap \tau'| = i} \E(I_\tau I_{\tau'}).
    \end{equation*}
    
    First, we bound the total contribution from the terms satisfying $|\tau\cap \tau'|=i$ for $i\in \{2,3,4\}$.
    Such pairs lie on a common lattice line and together contain exactly $8-i$ distinct points. For a given set $\sigma$ of $8-i$ collinear points, there are exactly $\binom{8-i}{i,\, 4-i,\, 4-i}$ pairs $(\tau, \tau')$ of collinear quadruples such that $\tau \cup \tau' = \sigma$. Thus $V_i = O(\E[Y^\mathrm{lin}_{8-i,T}])$, where $Y^\mathrm{lin}_{k,T}$ denotes the number of unordered collinear $k$-tuples in $Q \cap B_T$. By running through the proof of Lemma~\ref{lem:main}~(a) for $k$-tuples instead of $4$-tuples, we can see that $\E[Y^\mathrm{lin}_{k,T}] = O(\alpha^k 2^T)$ for any $k\geq 4$. We conclude that $\E[Y^\mathrm{lin}_{8-i,T}] = O(\alpha^{8-i} 2^T) = O(2^T)$ for $i \in \{2,3,4\}$. Therefore, we obtain
    \begin{equation}
    \label{eqn:V234}
        V_2 + V_3 + V_4 = O(2^T).
    \end{equation}
    Next, we bound the total contribution from the terms satisfying $|\tau \cap \tau'| = 1$. We have
    \begin{equation*}
        V_1 \le \sum_{x \in B_T} p_{x} \Gamma_T(x)^2 \lesssim (T+1)^4 \cdot \sum_{x \in B_T} p_{x} \Gamma_T(x) \asymp (T+1)^4 \E[Y_T^{\mathrm{lin}}] \lesssim (T+1)^4 2^T,
    \end{equation*}
    where the second and the last inequalities follow from Lemma~\ref{lem:onefix} and Lemma~\ref{lem:main}~(a), respectively. Adding this to \eqref{eqn:V234} and using the resulting estimate in \eqref{eq: second moment} yields the desired result.
\end{proof}

\section{Symmetric cyclic quadruples}
\label{sec:sym}

In this section, we prove parts (c) and (d) of Lemma~\ref{lem:main}. We begin by introducing some notation. For a primitive vector $v = (a,b) \in \mathbb Z^2$, let $v^\perp \defeq (-b,a)$. Note that $v^\perp$ is orthogonal to $v$ and it is primitive if and only if $v$ is. Suppose that $\{x, y, z, w\} \subset \mathbb{Z}^2$ is a symmetric cyclic quadrilateral, that is, an isosceles trapezium. For each isosceles trapezium, choose a reflection axis that interchanges its vertices in two pairs, using a fixed deterministic rule when more than one such axis exists. Suppose that reflection through the axis of symmetry interchanges $x$ with $y$ and $z$ with $w$. We define the \emph{axis direction} of the isosceles trapezium to be $v$, where $v^\perp$ is the direction of the line joining $x$ and $y$. Here, the direction of a lattice line is as defined at the start of Section \ref{sec:lin}. By construction, the chosen axis of symmetry has direction $v$. For an integer $T \ge 0$, let $\mathcal{S}_T$ denote the set of symmetric cyclic quadruples of points in $B_T$.

\begin{proof}[Proof of Lemma~\ref{lem:main}~(c)]
    By linearity of expectation and independence of $(\mathbf{1}\{x \in Q\})_{x \in B_T}$, we have
    \begin{equation}
    \label{eqn:EYTsym}
        \E[Y_T^{\mathrm{sym}}]
        = \sum_{\{x_1,x_2,x_3,x_4\}\in\mathcal{S}_T} p_{x_1} p_{x_2} p_{x_3} p_{x_4}.
    \end{equation}
    Fix $0 \le t_1 \le t_2 \le t_3 \le t_4 \le T$. We first bound the total contribution to the sum in \eqref{eqn:EYTsym} from terms such that $x_i \in R_{t_i}$ for $i \in [4]$. There are $|R_{t_1}| \le 2^{2t_1}$ choices for $x_1$. Fix a point $x_1 \in R_{t_1}$. Let $m \in \N$ and let $v$ be a primitive vector in $\Z^2$ with $\|v\|_\infty = m$. We wish to bound
    \begin{equation}
    \label{eq: S(x,v) defn}
        \tilde{S}(x_1, v) \defeq \sum p_{x_2} p_{x_3} p_{x_4},
    \end{equation}
    where the sum is over triples $\{x_2, x_3, x_4\} \subseteq B_T \setminus \{x_1\}$ such that $\{x_1, x_2, x_3, x_4\}$ forms an isosceles trapezium with axis direction $v$, and $x_i \in R_{t_i}$ for $2 \le i \le 4$. For any such triple $\{x_2, x_3, x_4\}$, let $x_j$ be the reflected mate of $x_1$ and let $\{x_k, x_\ell\}$ be the other reflected pair of vertices with
    \begin{equation*}
        k = k(j) \defeq \min (\{2,3,4\} \setminus \{j\}), \quad \text{ and } \quad \ell = \ell(j) \defeq \max (\{2,3,4\} \setminus \{j\}).
    \end{equation*}
    Then $x_j\in x_1+\Z v^\perp$ and $x_k \in x_\ell + \Z v^\perp$. It follows that $m \le \min \{\|x_1 - x_j\|_\infty,\|x_k-x_\ell\|_\infty \} \le 2^{\min\{t_j,t_\ell\}}$ and that there are at most $O(\frac{2^{t_j}}{m})$ choices for $x_j$. Further, note that
    \begin{equation*}
        x_k = \frac{x_1 + x_j}{2} + \frac{x_k + x_\ell - x_1 - x_j}{2} + \frac{x_k - x_\ell}{2} \in \left(\frac{x_1 + x_j}{2} + \Z \frac{v}{2} + \Z \frac{v^\perp}{2}\right)
    \end{equation*}
    since $\frac{x_k + x_\ell - x_1 - x_j}{2} \in \R v \cap (\frac{1}{2}\Z)^2 = \mathbb{Z}\frac{v}{2}$. Hence there are $O((2^{t_k})^2/m^2 + 1)$ choices for $x_k \in B_{t_k}$.
    Note that fixing $x_k$ determines $x_\ell$ uniquely. Since $p_{x_i} = O(\frac{\alpha}{2^{t_i}})$ for $2\leq i\le 4$, we obtain
    \begin{equation*}
        \tilde S(x_1, v) = O\left(\frac{\alpha^3}{m2^{t_2} 2^{t_3} 2^{t_4}} \sum_{j = 2}^4 2^{t_j} \left(1 + \frac{2^{2t_{k(j)}}}{m^2}\right) \mathbf{1}_{m \le 2^{t_{\ell(j)}}}\right).
    \end{equation*}
    Now since there are $O(m)$ choices for the primitive vector $v$ with $\|v\|_{\infty} = m$, it follows that the total contribution of terms such that $x_i\in R_{t_i}$ for $i\in [4]$ to the sum in \eqref{eqn:EYTsym} is at most 
    \begin{equation*}
        |R_{t_1}|\cdot\sum_{m=1}^\infty m \cdot\max_{x_1\in R_{t_1},\|v\|_\infty=m} p_{x_1}\tilde S(x_1,v)\lesssim\alpha^4 2^{t_1} \left( \sum_{j=2}^4 \frac{2^{t_j} 2^{t_{\ell(j)}}}{2^{t_2} 2^{t_3} 2^{t_4}} + \sum_{j=2}^4 \frac{2^{t_j} 2^{2t_{k(j)}}}{2^{t_2} 2^{t_3} 2^{t_4}} \sum_{m=1}^\infty \frac{1}{m^2}  \right) \lesssim\alpha^4 2^{t_1},
    \end{equation*}
    where the last inequality holds since $\max \{t_j + t_{\ell(j)}, t_{j} + 2t_{k(j)}\} \le t_j + t_{k(j)} + t_{\ell(j)} = t_2 + t_3 + t_4$. Finally, since for any given $t_1$, there are at most $(T - t_1 + 1)^3$ choices for $(t_2, t_3, t_4)$, it follows that
    \begin{equation*}
        \E[Y_T^{\mathrm{sym}}] \lesssim \sum_{t_1 = 0}^T \alpha^4 2^{t_1}(T-t_1+1)^3 \lesssim \alpha^4 2^T,
    \end{equation*}
    where the last inequality follows from the convergence of the series $\sum_{t\geq0}2^{-t}(t+1)^3$.
\end{proof}

Analogous to the definition of $\Gamma_T(x)$ in Section~\ref{sec:lin}, for $x \in B_T$, we define
\begin{equation*}
    \tilde{\Gamma}_T(x) \defeq \sum_{\{x, x_2, x_3, x_4\} \in \mathcal{S}_T} p_{x_2} p_{x_3} p_{x_4}.
\end{equation*}
Further, for distinct $x, y \in B_T$, we define
\begin{equation}
\label{eqn:Dxy}
    D_T(x, y) \defeq \sum_{\{x, y, x_3, x_4\} \in \mathcal{S}_T} p_{x_3} p_{x_4}.
\end{equation}
To prove Lemma~\ref{lem:main}~(d), we first establish the following pointwise bounds on $\tilde{\Gamma}_T(x)$ and $D_T(x,y)$.

\begin{lemma}
\label{lem:symonefix}
    For every $x \in B_T$, we have
    \begin{equation*}
        \tilde{\Gamma}_T(x) = O(\alpha^3 (T+1)^4).
    \end{equation*}
\end{lemma}

\begin{proof}
    Let $\tau = \{x, x_2, x_3, x_4\}$ be an isosceles trapezium in $\mathcal{S}_T$. Suppose that it satisfies $x_i \in R_{t_i}$ for each $i \in \{2, 3, 4\}$, with $t_2 \le t_3 \le t_4$. Let $v$ be the axis direction of $\tau$ and let $m \defeq \|v\|_\infty$. Further, let $x_j$ be the reflected mate of $x$ and let $\{x_k, x_\ell\}$ be the other reflected pair of vertices with
    \begin{equation*}
        k = k(j) \defeq \min (\{2,3,4\} \setminus \{j\}), \quad \text{ and } \quad \ell = \ell(j) \defeq \max (\{2,3,4\} \setminus \{j\}).
    \end{equation*}
    We say that $\tau$ is of type $1$ if $m \le 2^{t_k}$, of type $2$ if $m > 2^{t_k}$ and $j = 4$, and of type $3$ otherwise, namely, if $m > 2^{t_k}$ and $j \in \{2,3\}$. Fix $0 \le t_2 \le t_3 \le t_4 \le T$. For $z \in [3]$, define
    \begin{equation*}
        \tilde{\Gamma}_z(x) \defeq \sum p_{x_2} p_{x_3} p_{x_4},
    \end{equation*}
    where $\{x, x_2, x_3, x_4\} \in \mathcal{S}_T$ is of type $z$ and $x_i \in R_{t_i}$ for each $i \in \{2, 3, 4\}$. It suffices to show that $\tilde{\Gamma}_1(x) + \tilde{\Gamma}_2(x) + \tilde{\Gamma}_3(x) = O(\alpha^3 (T+1))$, since the result then follows by noting that there are at most $(T+1)^3$ choices for $(t_2, t_3, t_4)$.
    
    As in the proof of Lemma~\ref{lem:main}~(c), note that $1 \le m \le \|x_k - x_\ell\|_\infty \le 2^{t_\ell}$, and there are $O(m)$ choices for $v$ satisfying $\|v\|_\infty = m$. Further, note that for fixed $v$, the number of choices for $x_j$ is $O(2^{t_j}/m+1)$, and once $x_j$ is also fixed, the number of choices for $x_k$ is $O(2^{2t_k}/m^2+1)$. 

    For trapezia of type $1$, we have $2^{2t_k} \ge m^2$. So the number of choices for $x_k$ for fixed $v$ and $x_j$ is $O(2^{2t_k}/m^2)$. Therefore, we have
    \begin{equation*}
        \tilde{\Gamma}_1(x) \lesssim \sum_{j = 2}^4 \sum_{m=1}^{2^{t_{\ell(j)}}} m \cdot \frac{\alpha^3}{2^{t_2} 2^{t_3} 2^{t_4}} \left(1+\frac{2^{t_j}}{m}\right) \frac{2^{2t_{k(j)}}}{m^2} \lesssim \sum_{j = 2}^4 \alpha^3 \left((t_{\ell(j)}+1) \cdot \frac{2^{2t_{k(j)}}}{2^{t_2+t_3+t_4}} + \frac{2^{t_j+2t_{k(j)}}}{2^{t_2+t_3+t_4}}\right).
    \end{equation*}
    Noting that $t_{k(j)} \le t_{\ell(j)} \le T$ in the last expression above yields
    \begin{equation}
    \label{eqn:G1}
        \tilde{\Gamma}_1(x) \lesssim \alpha^3 (T+1).
    \end{equation}
    For trapezia of type $2$ and $3$, we have $2^{2t_k} < m^2$. So the number of choices for $x_k$ for fixed $v$ and $x_j$ is $O(1)$. Now, for trapezia of type $2$, we have $j = 4$ and $\ell(j) = 3$, and so
    \begin{equation}
    \label{eqn:G2}
        \tilde{\Gamma}_2(x) \lesssim \sum_{m=1}^{2^{t_3}} m \cdot \frac{\alpha^3}{2^{t_2} 2^{t_3} 2^{t_4}} \left(1+\frac{2^{t_4}}{m}\right) \lesssim \alpha^3 \left(\frac{2^{2t_3}}{2^{t_2+t_3+t_4}} + \frac{2^{t_3+t_4}}{2^{t_2+t_3+t_4}}\right) \lesssim \alpha^3.
    \end{equation}
    Finally, note that the axis direction $v$ is determined once $x_j$ is fixed. Note that there are at most $|R_{t_j}| \le 2^{2t_j}$ choices for $x_j$. For trapezia of type $3$, we have $j \le 3$ and $\ell(j) = 4$, and so
    \begin{equation}
    \label{eqn:G3}
        \tilde{\Gamma}_3(x) \lesssim \sum_{j = 2}^3 2^{2t_j} \cdot \frac{\alpha^3}{2^{t_2} 2^{t_3} 2^{t_4}} \lesssim \alpha^3.
    \end{equation}
    Adding the estimates \eqref{eqn:G1}, \eqref{eqn:G2}, and \eqref{eqn:G3} yields the claim.
\end{proof}

\begin{lemma}
\label{lem:symtwofix}
    For distinct $x, y \in B_T$, we have
    \begin{equation*}
        D_T(x, y) = O(\alpha^2 (T+1)).
    \end{equation*}
\end{lemma}

\begin{proof}
    Fix $0 \le t_3 \le T$. We wish to bound the contribution to the sum in \eqref{eqn:Dxy} from terms such that $x_i \in R_{t_i}$ for each $i \in \{3, 4\}$ with $t_3 \le t_4 \le T$. There are $|R_{t_3}| \le 2^{2t_3}$ choices for $x_3$. Moreover, once $x_3$ is fixed, there are at most $3$ choices for $x_4$ so that $x$, $y$, $x_3$ and $x_4$ form an isosceles trapezium. Finally, since $p_{x_3} p_{x_4} = O(\alpha^2 2^{-t_3-t_4}) = O(\alpha^2 2^{-2t_3})$, it follows that
    \begin{equation*}
        D_T(x, y) \lesssim \sum_{t_3 = 0}^T 2^{2t_3} \cdot \alpha^2 2^{-2t_3} \lesssim \alpha^2 (T+1),
    \end{equation*}
    thereby finishing the proof.
\end{proof}

\begin{proof}[Proof of Lemma~\ref{lem:main} (d)]
    For $\tau \in \mathcal{S}_T$, let $I_\tau$ denote the indicator of the event that all points in the quadruple $\tau$ are selected. Similarly as in the proof of Lemma~\ref{lem:main}~(b), we have
    \begin{equation}
    \label{eq:Var sym expand}
        \Var(Y_T^{\mathrm{sym}})
        \le \sum _{i=1}^4 V_i,
    \end{equation}
    where, for $i \in [4]$, $V_i$ is defined as
    \begin{equation*}
            V_i \defeq \sum_{|\tau \cap \tau'| = i} \E(I_\tau I_{\tau'}).
    \end{equation*}
    
    First, we bound $V_1$ as
    \begin{equation}
    \label{eqn:V1}
        V_1 \le \sum_{x \in B_T} p_{x} \tilde{\Gamma}_T(x)^2 \lesssim (T+1)^4 \cdot \sum_{x \in B_T} p_{x} \tilde{\Gamma}_T(x) \asymp (T+1)^4 \E[Y_T^{\text{sym}}],
    \end{equation}
    where the second inequality follows from Lemma~\ref{lem:symonefix}. Next, we bound $V_2$ as
    \begin{equation}
    \label{eqn:V2}
        V_2 \le \sum_{x\neq y} p_x p_y D_T(x,y)^2 \lesssim (T+1) \cdot \sum_{x\neq y} p_x p_y D_T(x,y) \asymp (T+1) \E[Y_T^{\text{sym}}],
    \end{equation}
    where the second inequality follows from Lemma~\ref{lem:symtwofix}. Last, note that for a given $\tau \in \mathcal{S}_T$, there are at most $3\binom{4}{3} = 12$ choices for $\tau' \in \mathcal{S}_T$ so that $|\tau \cap \tau'| \ge 3$ and so it follows that
    \begin{equation}
    \label{eqn:V34}
        V_3 + V_4 \lesssim \sum_{\tau \in \mathcal{S}_T} \E[I_\tau] = \E[Y_T^{\text{sym}}].
    \end{equation}
    
    Now using the estimates \eqref{eqn:V1}, \eqref{eqn:V2}, and \eqref{eqn:V34} in \eqref{eq:Var sym expand} gives
    \begin{equation*}
        \mathrm{Var}(Y_T^{\mathrm{sym}}) \lesssim (T+1)^4 \E[Y_T^{\mathrm{sym}}].
    \end{equation*}
    Finally, using Lemma~\ref{lem:main}~(c) in the right-hand side above yields the desired result.
\end{proof}

\section{Asymmetric cyclic quadruples}
\label{sec:asym}

Let $T \ge 0$ be an integer. In this section, our objective is to bound the expected number of asymmetric cyclic quadruples in $B_T$. We begin with the following definition. A \emph{pinned} asymmetric cyclic lattice quadrilateral refers to an unordered quadruple $S = \{0, u_2, u_3, u_4\}$ of distinct points in $\mathbb{Z}^2$ such that these points are vertices of an asymmetric cyclic quadrilateral. Let $\mathcal{A}$ denote the set of such pinned quadrilaterals. For $i \in \{2,3,4\}$, we denote the infinity norm of $u_i$ by $d_i$. We henceforth adopt the convention that $u_2$, $u_3$, and $u_4$ are ordered such that $d_2 \le d_3 \le d_4$. For $D > 0$, let $\mathcal{A}_D$ denote the set of pinned quadrilaterals in $\mathcal{A}$ with $d_4 \le D$. Huxley and Konyagin~\cite{HK} showed that the number of translation classes of asymmetric cyclic lattice quadrilaterals with circumradius at most $R$ is $O_\varepsilon(R^{2+18/29+\varepsilon})$. For each translation class, there are exactly four pinned quadrilaterals in $\mathcal{A}$, so the same result also holds for these. However, as noted in \cite[Section~4]{GGK}, Huxley and Konyagin's proof also implies the following stronger result.

\begin{lemma}
\label{lem:HK}
    For every $\varepsilon > 0$, we have $|\mathcal{A}_D| = O_\varepsilon(D^{2+18/29+\varepsilon})$.
\end{lemma}

For $\Delta > 0$, let $\mathcal{B}_{\Delta}$ denote the set of pinned quadrilaterals in $\mathcal{A}$ with $d_3 \le \Delta$. In the following lemma, we bound the size of $\mathcal{B}_{\Delta}$.

\begin{lemma}
\label{lem:count}
    For every $\varepsilon > 0$, we have $|\mathcal{B}_{\Delta}| = O_\varepsilon(\Delta^{4+\varepsilon})$.
\end{lemma}

\begin{proof}
    The result holds trivially for $\Delta < 1$ since $d_3 \ge 1$ for every pinned quadrilateral in $\mathcal{A}$ and so $|\mathcal{B}_\Delta| = 0$. So we may assume $\Delta \ge 1$. We first choose linearly independent $u_2, u_3 \in \mathbb{Z}^2 \setminus \{0\}$ with $\|u_2\|_\infty \le \|u_3\|_\infty \le \Delta$, giving $O(\Delta^2)$ choices for each. Now the number of choices for $u_4$ is bounded above by the number of lattice points on the circumcircle of the triangle with vertices $0$, $u_2$, and $u_3$, which is in turn bounded above by $O_\eps(\Delta^\eps)$. The last bound is seen by writing down the equation of the circumcircle and using the well-known bound (see~\cite[Equation~11.9]{Iwa}) on the number of lattice points on an ellipse in conjunction with the divisor bound~\cite[Theorem~317]{HW}. This follows a template similar to that of the proof of \cite[Lemma~3.1]{GGK}, so we omit the exact details here. Multiplying the number of choices for $u_2$, $u_3$, and $u_4$ yields the desired result.
\end{proof}

Let $S = \{0, u_2, u_3, u_4\} \in \mathcal{A}$. Consider the translate $S + v \defeq \{v, v + u_2, v + u_3, v + u_4\}$ of $S$ by a nonzero vector $v \in \mathbb{Z}^2$. We say that the translate $S + v$ is \emph{well-oriented} if
\begin{equation*}
    \|v\|_\infty = \min_{w \in S + v} \|w\|_\infty.
\end{equation*}
Let $V_S$ denote the set of $v \in \N^2$ such that the translate $S + v$ is contained in $\N^2$ and is well-oriented. Further, let $Z(S)$ denote the number of well-oriented translates of $S$ contained inside the random set $Q$. In the following lemma, we bound the expected value of $Z(S)$.

\begin{lemma}
\label{lem:shape}
    For every pinned quadrilateral $S = \{0, u_2, u_3, u_4\} \in \mathcal{A}$, we have
    \begin{equation*}
        \E Z(S) \lesssim \alpha^4 \left(\frac{\log d_3 + 1}{d_3 d_4}\right).
    \end{equation*}
\end{lemma}

\begin{proof}
    Set $u_1 = 0$. By Tonelli's theorem, we have
    \begin{equation}
    \label{eqn:exp}
        \E Z(S) = \sum_{v \in V_S} \prod_{i=1}^4 \frac{\alpha}{\|v + u_i\|_\infty} = \alpha^4 \sum_{v \in V_S} \prod_{i=1}^4 \frac{1}{\|v + u_i\|_\infty}.
    \end{equation}
    Let $v \in V_S$. Define $r(v) \defeq \|v\|_\infty$ and $x_i(v) \defeq v + u_i$ for $i \in [4]$. Since $S + v$ is well-oriented, we have $\|x_i(v)\|_\infty \ge r(v)$ for every $i \in [4]$. Further, for $i \in \{2,3,4\}$, the triangle inequality also gives $\|x_i(v)\|_\infty \ge d_i - r(v)$, which implies
    \begin{equation*}
        \|x_i(v)\|_\infty \ge \max\{r(v),d_i-r(v)\} \ge \frac{1}{2} \max\{r(v),d_i\}.
    \end{equation*}
    Using the lower bound $r(v)$ for $\|x_1(v)\|_\infty$ and $\|x_2(v)\|_\infty$, $\frac{1}{2} \max\{r(v), d_3\}$ for $\|x_3(v)\|_\infty$, and $d_4/2$ for $\|x_4(v)\|_\infty$, we obtain
    \begin{equation*}
        \prod_{i=1}^4 \frac{1}{\|x_i(v)\|_\infty} \le \frac{4}{r(v)^2d_4\max\{r(v),d_3\}}.
    \end{equation*}
    Using the above estimate together with the bound $|\{v \in V_S: \|v\|_\infty = r\}| \le 2r$ in \eqref{eqn:exp}, we obtain
    \begin{equation*}
        \E Z(S) \le \alpha^4 \sum_{r = 1}^\infty (2r) \cdot \frac{4}{r^2 d_4 \max\{r,d_3\}} = \frac{8 \alpha^4}{d_4} \sum_{r = 1}^{\infty} \frac{1}{r\max\{r,d_3\}}.
    \end{equation*}
    Breaking the sum on the right-hand side into two parts, corresponding to $r \le d_3$ and $r > d_3$, yields
    \begin{equation*}
        \E Z(S) \le \frac{8 \alpha^4}{d_4} \left(\frac{1}{d_3} \sum_{r = 1}^{d_3} \frac{1}{r} + \sum_{r = d_3 + 1}^{\infty} \frac{1}{r^2}\right) \lesssim \alpha^4 \left(\frac{1 + \log d_3}{d_3 d_4}\right).
    \end{equation*}
    This finishes the proof of the lemma.
\end{proof}

\begin{proof}[Proof of Lemma~\ref{lem:main} (e)]
    We wish to bound the expected number of asymmetric cyclic quadruples in $Q \cap B_T$. Each such quadruple corresponds to a well-oriented translate of some pinned quadrilateral in $\mathcal{A}_{2^T}$ that appears in $Q$. Therefore, we have
    \begin{equation}
    \label{eqn:amain}
        \E Y^\mathrm{asym}_T \le \sum_{S \in \mathcal{A}_{2^T}} \E Z(S).
    \end{equation}
    For $0 \le j \le i \le T$, define $\mathcal{C}_{i,j} \defeq (\mathcal{A}_{2^i} \setminus \mathcal{A}_{2^{i-1}}) \cap (\mathcal{B}_{2^j} \setminus \mathcal{B}_{2^{j-1}})$. Then $(\mathcal{C}_{i,j})_{0 \le j \le i \le T}$ is a partition of $\mathcal{A}_{2^T}$. Using this partition in \eqref{eqn:amain}, we obtain
    \begin{equation}
    \label{eqn:ijbound}
        \E Y^\mathrm{asym}_T \le \sum_{i=0}^T \sum_{j=0}^i \sum_{S \in \mathcal{C}_{i,j}} \E Z(S).
    \end{equation}
    For each $\eps > 0$, Lemmas~\ref{lem:HK} and \ref{lem:count} imply
    \begin{equation}
    \label{eqn:Cij}
        |\mathcal{C}_{i,j}| \le \min \{|\mathcal{A}_{2^i}|, |\mathcal{B}_{2^j}|\} \lesssim_\eps \min\{(2^i)^{76/29+\eps}, (2^j)^{4+\eps}\} \lesssim_\eps 2^{i(57/29+3\eps/4) + j(1+\eps/4)},
    \end{equation}
    where the last inequality follows from the fact that $\min \{a,b\} \le a^{3/4} b^{1/4}$ for any $a, b > 0$. Also, for each $S \in \mathcal{C}_{i,j}$, we have $2^{j-1} < d_3 \le 2^j$ and $2^{i-1} < d_4 \le 2^i$. Therefore, Lemma~\ref{lem:shape} implies
    \begin{equation}
    \label{eqn:EZS}
        \frac{\E Z(S)}{\alpha^4} \lesssim \frac{j \log 2 + 1}{2^{i-1} 2^{j-1}} \lesssim \frac{j+1}{2^{i+j}}.
    \end{equation}
    Using \eqref{eqn:Cij} and \eqref{eqn:EZS} in \eqref{eqn:ijbound} and then using the estimate $T+1 \lesssim_\eps 2^{\eps T/3}$, we obtain
    \begin{align*}
        \E Y^\mathrm{asym}_T \lesssim_\eps \alpha^4 \sum_{i=0}^T \sum_{j=0}^i 2^{28i/29 + \eps(3i+j)/4} (j+1) \lesssim_\eps \alpha^4 2^{28T/29 + \eps T} (T+1)^3 \lesssim_\eps \alpha^4 2^{T(28/29 + 2\eps)}.
    \end{align*}
    Choosing $\eps \in (0, (1/29-0.02)/2)$ gives the desired result.
\end{proof}

\bibliographystyle{amsplain}
\bibliography{references}

\end{document}